\documentclass[A4paper, 12pt]{article}

\usepackage{hyperref}
\usepackage[utf8]{inputenc}
\usepackage[OT2,T1]{fontenc}
\usepackage[french, english]{babel}
\usepackage[osf]{mathpazo}
\usepackage{amsmath}
\usepackage{graphicx}

\usepackage{color}
\definecolor{FOB}{rgb}{0.,0.,0.7}
\definecolor{FOR}{rgb}{0.65,0.,0.}
\definecolor{FOfond}{rgb}{0.98,0.98,0.91}

\pagecolor{FOfond}

\usepackage{sectsty}
\allsectionsfont{\color{FOR}}

\usepackage{amssymb}
\usepackage{authblk}

\def\cC{{\mathcal{C}}}
\def\cH{\mathcal{H}}
\def\cP{\mathcal{P}}
\def\cQ{{\mathcal{Q}}}

\def\RR{{\mathbb{R}}}

\def\dd{{\mathrm{d}}}

\def\cF{\mathcal{F}}

\def\bull{\vrule height .9ex width .8ex depth -.1ex }

\newcounter{thenum}
\def\texttheo{\relax}
\newenvironment{theorem}{\medbreak\refstepcounter{thenum}
\noindent\textsc{Theorem} %
\thethenum. \texttheo ---  \it  }{\rm }
\newenvironment{e-proposition}{\medbreak\refstepcounter{thenum}
\noindent\textsc{Proposition} \thethenum. ---  \it  }{\rm }

\newenvironment{e-definition}{\medbreak\refstepcounter{thenum}
\noindent\textsc{Definition} \thethenum. ---  \it  }{\rm }

\newenvironment{remark}{\medbreak\refstepcounter{thenum}\noindent{\it Remark} %
\thethenum. --- }{}

\newenvironment{e-rem}{\medbreak\refstepcounter{thenum}{}%
 \thethenum) }{}

\newenvironment{e-ex}{\medbreak\refstepcounter{thenum}{}%
 \thethenum) }{}
\newenvironment{proof}{\smallbreak\noindent{\sc Proof.} --- \rm}{\quad\bull\smallskip\rm}

\begin{document}

\begin{center}{\LARGE\parindent=0pt{\color{FOR}
      \selectlanguage{english} Does the problem of the Nile have a
      solution?
      \medskip
      
      \Large\selectlanguage{french} Le problème du Nil a-t-il une solution?
      
}}
\end{center}
\vskip2cm

\hbox to \hsize{\parindent =0pt\hbox to 2.5cm{\hfill}\hss
  \vbox{\hsize  = 7cm
    {\large Guillaume \textsc{Chèze}}
  \bigskip

  Université de  Toulouse,
  
  INSA Toulouse, CNRS,

  IMT, Toulouse,

  France

\smallskip

{\tiny guillaume.cheze@math.univ-toulouse.fr}} \hss\hss
\vbox{\hsize
    = 7cm {\large François \textsc{Ollivier}} 
\bigskip

LIX, UMR CNRS 7161 

École polytechnique 

91128 Palaiseau \textsc{cedex}

France

\smallskip

{\tiny francois.ollivier@lix.polytechnique.fr}
}\hss}
\vskip 0.3cm

\begin{center}\parindent =0pt 8 janvier 2022
\end{center}
\vfill

{\small \selectlanguage{english}
\hbox to \hsize{\hss\vbox{\hsize= 5.8cm\selectlanguage{english}
    \noindent \textbf{Abstract.}\spaceskip=1ex plus 5pt minus 2pt    
The problem of the Nile has been stated by R.~Fisher in 1936. This
problem concerns the fair division of an agricultural land and is the
eponymous problem of a general class of problems on measures and
similar regions. In 1938, W.~Feller has shown that the general problem
has no solution. However, we will show that under a natural
hypothesis, the original problem of the Nile is solvable with
connected plots. 
  \smallskip

  \noindent \textbf{Keywords:} Problem of the Nile, Fair cake cutting.
}
\hss\hss\hss

\vbox{\hsize= 6.2cm\selectlanguage{french}
  \noindent \textbf{Résumé.}\spaceskip=1 ex plus 5 pt minus 2pt  
Le problème du Nil a été formulé par R.~Fisher en 1936. Ce problème
porte sur le partage équitable d'une terre agricole et a donné son nom
à une classe générale de problèmes concernant les mesures et les
régions analogues. En 1938, W.~Feller a démontré que le problème
général n'admettait pas de solution. Nous montrerons toutefois que,
sous une hypothèse naturelle, le problème initial du Nil peut être
résolu en utilisant des parcelles connexes.
  \smallskip

  \noindent \textbf{Mots-clés:} Problème du Nil, Partages équitables.
}\hss}}

\selectlanguage{english}
\smallskip

\noindent \hphantom{x}\hbox to 0.1cm{\hfill}{\selectlanguage{french}\small
  \textbf{Classification AMS:} {91B32}.}
\eject

\hbox to\hsize{\hss \textit{Qui aquam Nili bibit rursus
    bibet.}\hskip\parindent}
\hbox to\hsize{\hss \textit{He who drinks the water of the Nile will
    drink it again.}\hskip\parindent}
\bigskip

\section*{Introduction}
In 1936, Ronald Fisher wrote an article entitled ``Uncertain
Inference'' \cite{Fisher} in which he proposed to ``look at both the
past and the future''. At the end of the article, he posed a problem
illustrated with the following example:

\emph{The agricultural land of a pre-dynastic Egyptian village is of
unequal fertility.  Given the height to which the Nile will rise, the
fertility of every portion of it is known with exactitude, but the
height of the flood affects different parts of the territory
unequally.  It is required to divide the area, between the several
households of the village, so that the yields of the lots assigned to
each shall be in pre-determined proportion, whatever may be the height
to which the river rises.}\\

The problem asked by Fisher has thus the following mathematical formulation.\\
Consider a plot $\mathcal{P} \subset \RR^2$, and  $\alpha\in\RR^{n}$,
where $\alpha_{i} \geq 0$ with $\sum_{i=1}^n \alpha_{i}=1$.  Can we divide $\cP$, in order to have $\cP=\sqcup_{i=1}^n \cP_{i}$ and  for any value $h$, $\mu_{h}(\cP_{i})=\alpha_{i} \mu_{h}(\cP)$, where
$\mu_{h}(\cP_{i})$ is the value of $\cP_{i}$ when the height of the
flood is equal to $h$?\\
%!!!! Chercher scihub Fisher:  Quelques remarques sur l'estimation en statistique. Biotypologie, 6: 153-
%158.!!!

In 1938, William Feller~\cite{Feller} studied the general problem posed by Fisher and
gave a negative answer. The aim was to find out
whether for a measure $\mu_{\theta_1, \theta_2 \ldots, \theta_n}$ on
$\RR^n$ depending on the parameters $\theta_1, \theta_2 \ldots,
\theta_n$, it is always possible to find a region $A$ such that
$\mu_{\theta_1, \theta_2 \ldots, \theta_n}(A)=\alpha$ for all values
of the parameters $\theta_1$, \ldots, $\theta_n$, when $\alpha$ is
fixed. If such a region exists, it is said to be ``similar to the
sample space $\RR^n$". Feller then proved that it is not always
possible to find such a region. In particular, he showed that it is
impossible when we consider the measure $\mu_{\theta_1, \theta_2
  \ldots, \theta_n}$ whose density is given by
$$\dfrac{1}{(\sqrt{2\pi})^n}e^{-\frac{(x_1-\theta_1)^2+\cdots+(x_n-\theta_n)^2}{2}}.$$

Feller's result might suggest that there is no solution to the problem
of the Nile. 
We can indeed expect, if we use a sufficiently
general model, that the problem of the Nile has no solution.\\

 Later, in 1961, Dubins and Spanier, in their remarkable article on
 cake-cutting~\cite{Dubins-Spannier}, referred to the problem of the
 Nile and wrote the following:\\ \emph{``Each year, the Nile would
 flood, thereby irrigating or perhaps devastating parts of the
 agricultural land of a predynastic Egyptian village. The value of
 different portions of the land would depend upon the height of the
 flood. In question was the possibility of giving to each of the n
 residents, piece of land whose value would be 1/n of the total land
 value, no matter what the height of the flood."\\ The problem as
 described above allows an infinite number of flood heights and in
 such a case, as shown by Feller \cite{Feller}, need not have a
 solution.  }\\

``Need not\dots'' these cautious words left open the possible
 existence of a solution, at least with some mathematical models.  In
 this note, we will show that if we consider the problem of the Nile,
 it may actually have a solution under some natural simplification
 hypotheses.

%%%%%%%%%%%%%%%%%%%%%%%%%%%%%%%%%

\section{Modeling the problem}
When we consider the problem of the Nile, we can restrict our study to
one bank of the river. Therefore, the problem of the Nile can be
illustrated by the following picture, see Figure~\ref{fig:Nile},
showing a plot $\mathcal{P}$ bounded on the west by the Nile.\\
In antiquity, a traditional orientation was to look east, so that north was ``left'' and south ``right''. In Figure~\ref{fig:Nile}, we use this orientation. \\ 

\begin{figure}[h!]
\centering
\includegraphics[scale=0.9]{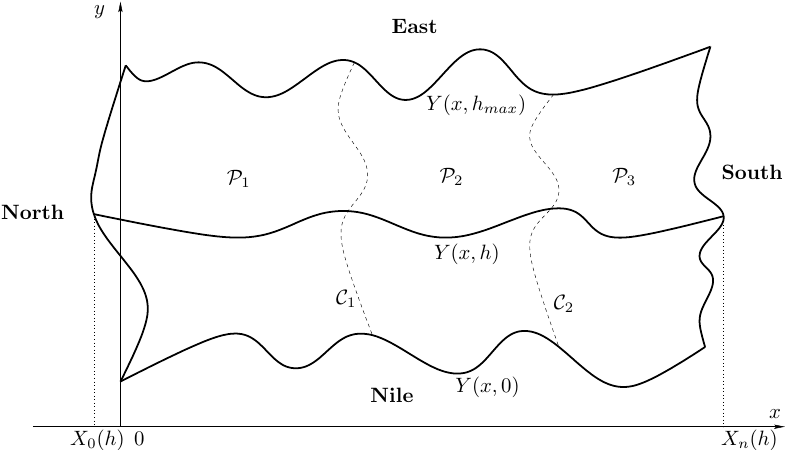}
\caption{Division of a plot $\cP$ between 3 households}\label{fig:Nile}
\end{figure}

In order to describe the plot, we need to introduce some functions.
In the following, the northern and the southern limits will be represented thanks to
continuous functions $X_0(h)$ and $X_n(h)$ from $[0,h_{\rm max}]$ to
$\RR$. The parameter $h$ represents the height of the flood and
$h_{\rm max}$ the maximum height of a flood.\\ Thanks to these
functions, we introduce the set
$$\mathcal{P}'=\{ (x,h) \in \RR^2  \,|\, 0\leq h \leq h_{\rm max}, X_0(h) \leq x \leq X_n(h)\}.$$

Now, the bank of the Nile when the height of the flood is equal to $h$
is given by the function $Y(x,h)$ defined on $\mathcal{P}'$. In
particular, the western and the eastern limits of the plot
correspond respectively to $Y(x,0)$ and $Y(x,h_{\rm
  max})$.

Furthermore, the previous notations mean that we assume
that the northern limit and southern limits are respectively defined
by parametric curves
\begin{equation}\label{eq:bord}
\cC_{0}=\Big(X_{0}(h),Y\big(X_{0}(h),h\big)\Big)\>\hbox{ and }\>\cC_{n}=\Big(X_{n}(h),Y\big(X_{n}(h),h\big)\Big),
\end{equation}

where $n$ is the number of plots in the requested
partition. With such a notation, $\cC_{n}$ is the southern
border of the last plot.\\

Furthermore, we suppose that 
\begin{equation}\label{eq:Ygt0}
  \dfrac{\partial Y}{\partial h}(x,h)>0.
\end{equation}
  This means that when the
flood height $h$ increases, the parameter $y$ on our drawing
increases too. In practical terms, this means that as the flood height $h$
increases, the Nile overflows its banks more and more.  Therefore, $Y$
allows us to obtain a change of variables
$(x,y)=\varphi(x,h)$:
\begin{equation}
  \begin{array}{rll}\label{eq:chgt_var_1}
    \varphi: \mathcal{P}' & \longrightarrow & \cP\\
           (x,h) & \longmapsto &  \big(x, Y(x,h) \big)=(x,y),
  \end{array}
\end{equation}
where $\cP$ is the plot to be divided.\\

We denote by  $\varphi^{-1}$ the inverse of this diffeomorphism

\begin{eqnarray*}\varphi^{-1}: \mathcal{P}& \longrightarrow &  \mathcal{P}'\\ 
(x,y) & \longmapsto & 
  \big(x, H(x,y) \big)=(x,h)
\end{eqnarray*}

The function $H(x,y)$ therefore represents the flood height at which the Nile will cover the point with coordinates $(x,y)$.\\

Now, we are going to explain how we model the value of a plot $\mathcal{Q}$. A simple way to measure the value of a plot $\mathcal{Q}$ is to consider its area. However, 
 we can also measure the value of a plot as a function of the time it
was flooded. In this way, let $\cH \in\cC^{0}\big([0,T],[0,h_{\max}]\big)$ be the
function that models the height of the water, $T$ being the maximal
duration of the flooding. Therefore, in this general situation, the value of a plot is given thanks to the measure $\mu_{\cH}$ defined in the following way:
\begin{equation}\label{eq:mu_h}
\mu_{\cH}(\mathcal{Q}) =\iint_{\mathcal{Q}} D_{\cH}(x,y) dxdy,
\end{equation}

where the density $D_{\cH}$ at
point $(x,y)$, is given by
\begin{equation}\label{eq:D}
D_{\cH}(x,y)=D(x,y,\cH)=G(x,y)F\big(H(x,y),\cH\big),
\end{equation}
where $G$ is a continuous function from $\mathcal{P}$ to $\RR^{+}$ 
and $F$ is a function
from\\ $[0,h_{\rm max}]\times\cC^{0}\big([0,T],[0,h_{\rm max}]\big)$ to
$\RR^{+}$ such that for all $\cH$, $F(\cdot,\cH)$ is continuous.\\

 The main
justification for this assumption is that the Nile shores are often
flat, so that all points will be affected in the same way
by other natural events such as the sun, the wind,  \dots\ The main action is
then the way in which they will be affected by the flood by receiving
the fertile silt, that mostly depend on their height. Therefore the
value of the land at two points of the same altitude can legitimately
be considered equal, which leads to consider $G(x,y)$ to be a
constant. However our construction will work for any function $G$.\\

For example, if we assume that the value of a plot of land is proportional to the length of time
it remained flooded, then we have  
$$F(H(x,y),\cH)=\int_{0}^{T}
\theta\big(\cH(t)-H(x,y)\big)\dd t,$$
 where $\theta$ stands for the Heaviside
function: $\theta(x)=1$ if $x\ge0$ and $\theta(x)=0$ if $x<0$.\\
 But we could use other functions, such as 
 $$F(H(x,y),\cH)=\int_{0}^{T} \max\big(\cH(t)-H(x,y),0\big)\dd t.$$ In
 this latter case, the value of a plot of land depends on the amount
 of water that was previously above it.

In what follows, we will show that 
it is possible to divide the territory into $n$ parcels $\cP_1$, \ldots,
$\cP_{n}$ with constant
predetermined proportion of values, \textit{i.e.} not depending on the
choice of $\cH$. Moreover, we will prove  that the boundary between  $\cP_{i}$ and  $\cP_{i+1}$  will be given by a parametric
curve $\mathcal{C}_{i}$, see Figure~\ref{fig:Nile}.

\begin{theorem}\label{th:theo_v1}
With the previous notations, for any $n$-uple $\alpha\in\RR^{n}$,
  where $\alpha_{i} \geq 0$ and $\sum_{i=1}^n \alpha_{i}=1$, there
  exists a partition $\cP=\sqcup_{i=1}^n \cP_i$, where
  the sets $\cP_{i}$ are connected and such that for all $\cH \in \cC^{0}([0,T])$ we have
$$\forall 1\le i\le n,\>
\mu_{\cH}(\cP_{i})=\alpha_{i} \mu_{\cH}(\cP).$$
Furthermore, the sets $\cP_i$ are defined by
\begin{equation}\label{eq:x_1}
\cP_i=\{ (x,y) \,|\, X_{i-1}(h)\le x\le X_{i}(h), \textrm{ and } h=H(x,y) \in[0,h_{\rm max}]\},
\end{equation}
where the functions %\footnote{We recall that $X_{0}$ and $X_{n}$ have
%already been defined as borders of $\cP$ (see eq.~\eqref{eq:bord}).}
$X_{i}:[0,h_{\rm max}]\mapsto\RR$, for $1<i<n$,
are continuous and satisfy
\begin{equation}\label{eq:x_2}
  X_{0}(h)\le X_{1}(h) \le\cdots\le X_{n-1}(h)\le X_{n}(h).
\end{equation}
\end{theorem}

%%%%%%%%%%%%%%%%%%%%%%%%%%%%%%%%%%%%
\section{Proof of Theorem~\ref{th:theo_v1}}

Before giving the proof in detail, we will provide an intuitive explanation.
Suppose that $\cP=[0,1] \times [0,1]$, and that $H(x,y)=y$ or, equivalently, $Y(x,h)=h$. This means that $\cP$ is a rectangle and that for all possible flood heights, the bank of the Nile is always a straight line parallel to the line $x=0$.\\
Let us further assume that the value of a point $(x,y)$ depends only on $y$ because all points with the same abscissa will have been covered in the same way during the flood. Then the value of a parcel $\cQ$ is given by an expression of the form $\mu_{\cH}(\cQ)=\iint_{\cQ} \mathcal{F}(y,\mathcal{H}) \dd x \dd y$.
We then see that the plot $\cP_i=[a_i,a_i+\alpha_i]\times [0,1]$ has the following value
$$\mu_{\cH}(\cP_i)= \int_0^1 \int_{a_i}^{a_i+\alpha_i} \mathcal{F}(y,\mathcal{H}) \dd x \dd y= \alpha_i \int_0^1\int_0^1 \mathcal{F}(y,\mathcal{H}) \dd x \dd y=\alpha_i \mu_{\cH}(\cP).$$
Thus, if we set $\alpha_0=0$, $a_i=\sum_{k=0}^{i-1} \alpha_i$, and $b_i=a_i+\alpha_i$, then $\sqcup_{i=1}^n [a_i,b_i] \times [0,1]$ gives the desired partition in this simple case.\\

In the general case, we use a change of variables to follow the same idea.\\

\begin{proof}[Proof of Theorem~\ref{th:theo_v1}]
First, we remark that thanks to the change of variables $\varphi$, see~\eqref{eq:chgt_var_1}, for all measurable subspace $\cQ \subset \cP$, we have

$$
\mu_{\cH}(\cQ)=\iint_\cQ D(x,y,\cH) \dd x\dd y=\iint_{\varphi^{-1}(\cQ)}  \tilde{D}(x,h,\cH) \dd x\dd h,
$$
where $\tilde D(x,h,\cH):=\tilde G(x,h)F(h,\cH)$, with $\tilde G(x,h):=G\big( x, Y(x,h)\big) \dfrac{\partial
Y}{\partial h}(x,h)$. As $G(x,y)>0$ and $\dfrac{\partial
Y}{\partial h}(x,y)>0$, see~\eqref{eq:Ygt0}, we have:
\begin{equation}\label{eq:Ggt0}
  \tilde{G}(x,h)>0.
\end{equation}

This allows us to define
the function $\tilde X$ in the following way:
\begin{equation}\label{eq:tildex}
  \tilde X(x,h)=\frac{\int_{X_{0}(h)}^{x} \tilde G(\xi,h) \dd \xi}
                     {\int_{X_{0}(h)}^{X_{n}(h)} \tilde G(\xi,h) \dd \xi}.
\end{equation}
By formula \eqref{eq:Ggt0}, for all $h\in[0,h_{\max}]$,
$\tilde{X}(\cdot,h)$ is a strictly increasing function of $x$, such
that $\tilde{X}(X_{0}(h),h)=0$ and $\tilde{X}(X_{n}(h),h)=1$, for any $0\le h\le
h_{\max}$. This defines the following diffeomorphism
\begin{eqnarray*}
\psi: \cP'& \longrightarrow & \cP''\\
(x,h) & \longmapsto & \big( \tilde{X}(x,h), h\big)=(\tilde{x},h)
\end{eqnarray*}
where
 %$\cP'= \varphi^{-1}(\cP)=\{(x,h) \in \RR^2 \,|
%\, 0 \leq h \leq h_{\rm max}, \, X_0(h) \leq x \leq X_n(h) \}$ 
%and
$\cP''=\psi\big(\varphi^{-1}(\cP)\big)=\{ (\tilde{x},h) \in \RR^2 \, | \, 0 \leq h \leq h_{\rm
  max}, \, 0 \leq \tilde{x}\leq 1\}$.\\

\begin{remark} If $G(x,y)=1$ then $\tilde
  X(x,h)=\dfrac{x-X_0(h)}{X_n(h)-X_0(h)}$. Using the function $\psi$
  is thus a natural way to transform $\cP'$ into a rectangle in order
  to apply the idea explained in the beginning of this section.
\end{remark}
  
By construction, the Jacobian determinant of $\psi$ is equal to
$\tilde{G}(x,h)/(\int_{X_{0}(h)}^{X_{n}(h)} \tilde{G}(\xi,h) \dd \xi)$. Therefore, the
change of variables $\psi$ gives for all plots $\cQ \subset
\cP$,
\begin{equation*}
  \mu_{\cH}(\cQ)=\iint_{\varphi^{-1}(\cQ)} \tilde G(x,h)F(h,\cH) \dd x\dd h=
  \iint_{\psi(\varphi^{-1}(\cQ))} \left(\int_{X_{0}(h)}^{X_{n}(h)} \tilde
  G(\xi,h) \dd \xi\right) F(h,\cH)\dd\tilde x\dd h.
\end{equation*}

We set $\mathcal{F}(h,\cH)=\left(\int_{X_{0}(h)}^{X_{n}(h)} \tilde
  G(\xi,h) \dd \xi\right) F(h,\cH)$, and this gives
  \begin{equation}\label{eq:tildeD}
  \mu_{\cH}(\cQ)=\iint_{\varphi^{-1}(\cQ)} \tilde G(x,h)F(h,\cH) \dd x\dd h=
  \iint_{\psi(\varphi^{-1}(\cQ))} \mathcal{F}(h,\cH)\dd\tilde x\dd h.
\end{equation}

Now, in order to construct the plot $\cP_i$, we set: $\alpha_0=0$, $a_i=\sum_{k=0}^{i-1} \alpha_i$, and  $b_i=a_i+\alpha_i$.
This allows us to define
$$\cP''_i=\{ (\tilde{x},h) \,|\,  a_i \leq \tilde{x} \leq b_i, 0 \leq h \leq h_{\rm max} \},$$ 
and  then
$$\cP_i=\varphi\big(\psi^{-1}(\cP''_i)\big).$$
We can remark that the sets $\cP_i$ are independent of $\mathcal{H}$.
Furthermore, this definition means that the continuous functions $X_i$ given in the theorem are defined thanks to the equality
$\int_{X_0(h)}^{X_i(h)} \tilde{G}(\xi,h)\dd\xi=b_i \int_{X_0(h)}^{X_n(h)} \tilde{G}(\xi,h)\dd\xi$. The inequality \mbox{$X_i(h) \leq X_{i+1}(h)$} is thus straighforward.

At last, we have
\begin{equation*}
\begin{split}  
\mu_{\cH}(\cP_{i})&=\iint_{\cP''_i} \mathcal{F}(h,\cH) \dd\tilde x\dd h =\int_{0}^{h_{\rm max}}
\int_{a_i}^{b_i}  \mathcal{F}(h,\cH)\dd\tilde
x\dd h\\
&=\int_{0}^{h_{\rm max}} \mathcal{F}(h,\cH) \int_{a_i}^{b_i} \dd\tilde
x\dd h =\int_{0}^{h_{\rm max}}
\alpha_i \mathcal{F}(h,\cH)  \dd h\\
&=\alpha_i \int_{0}^{h_{\rm max}} \mathcal{F}(h,\cH)  \int_{0}^1\dd\tilde x\dd h =\alpha_i \int_{0}^{h_{\rm max}}
\int_{0}^1 \mathcal{F}(h,\cH) \dd\tilde
x\dd h\\
&=\alpha_{i}\iint_{\cP''} \cF(h,\cH) \dd\tilde
x\dd h=\alpha_{i}\mu_{\cH}(\cP).
\end{split}
\end{equation*}
\end{proof}

\section*{Conclusion}
    
We have seen how the problem of the Nile can be solved, under some
reasonable hypotheses.  The reason of this situation is that the
problem of the Nile is more restricted than the problem considered by
Feller.

We consider indeed only one parameter $h$, the height of the
flood. Its effect mostly depends on the altitude of the considered
point that we take as a new coordinate function. If we wish to take
into account another parameter, for example the time $t$ of exposure
to the sun, in generic situations the effect of the flooding and of
the sun can be locally unique. This means that we can no longer apply
our strategy that is to cut the level curves, using the fact that the
new coordinate function $\tilde X$ has no influence on the relative
value of the ground. Indeed, the influence of sunshine can depend on
it. 

Under our simplification hypotheses, the Nile problem is therefore a
special case of Feller's problem for which a solution is
possible. Similar results considering more parameters, would
necessitate to parcel out a space of higher dimension, strictly
greater that the number of parameters.

 \bibliographystyle{plain}
 \bibliography{biblio_Nile}

\end{document}